\documentclass[onefignum,onetabnum]{siamart251216}
\usepackage{amsfonts}

\usepackage{bm}
\usepackage{algorithm}
\usepackage{algpseudocode}
\usepackage{url}

\newcommand{\dd}{\,\mathrm{d}}
\newcommand{\ii}{\mathrm{i}}
\newcommand{\dt}{\Delta t}
\newcommand{\dV}{\Delta V}

\newcommand{\kappaC}{\kappa}
\newcommand{\xx}{\bm{x}}
\newcommand{\vv}{\bm{v}}

\newcommand{\PP}{\bm{P}}
\newcommand{\AV}{\bm{A}}
\newcommand{\UU}{\bm{U}}
\newcommand{\JJ}{\bm{J}}
\newcommand{\EE}{\bm{E}}
\newcommand{\BB}{\bm{B}}
\newcommand{\kk}{\bm{k}}

\newcommand{\grad}{\nabla}
\newcommand{\curl}{\nabla\times}
\newcommand{\divg}{\nabla\cdot}

\newcommand{\curlh}{\nabla_h\times}
\newcommand{\divgh}{\nabla_h\cdot}

\newcommand{\DA}{\mathsf{D}}
\newcommand{\Sbar}{\overline{S}}
\newcommand{\vbar}{\overline{\vv}}
\newcommand{\Ubar}{\overline{\UU}}
\newcommand{\gphibar}{\overline{\grad\phi}}

\newcommand{\Etot}{\mathcal{E}}

\title{An Energy-Conserving Unstaggered Electromagnetic-Potential Particle-in-Cell Method, Part I: Non-relativistic Generalized-Momentum Formulation}

\author{Andrew J. Christlieb\thanks{Department of Computational Mathematics, Science and Engineering, Michigan State University, East Lansing, MI.}\and Luis Chac\'{o}n\thanks{Los Alamos National Laboratory, Los Alamos, NM}\and Sining Gong\thanks{Corresponding author. Department of Computational Mathematics, Science and Engineering, Michigan State University, East Lansing, MI}(\email{gongsini@msu.edu}).}

\headers{Energy-Conserving Potential PIC}{A. J. Christlieb, L. Chac\'{o}n, and S. Gong}

\usepackage{subcaption}

\begin{document}
\maketitle

\begin{abstract}
We develop an unstaggered, potential-based particle-in-cell method for the nonrelativistic Vlasov--Maxwell system in the Lorenz gauge.  The field update is written as a Crank--Nicolson discretization of first-order wave systems for the scalar potential, the vector potential, and their time derivatives.  The charge density is not deposited directly; instead it is advanced from the discrete continuity equation using the current deposited from the particles.  
This opens up algorithmic flexibility with a range of innovation, including unstaggered mesh layouts that preserve the Lorenz gauge and Gauss’s law at the discrete level. In the potential formulation, this source ordering also permits preservation of the Lorenz gauge and Gauss's law at the discrete level.   
To extend the paradigm to an energy-conserving formulation, we introduce a consistent orbit-averaged scatter, gather and particle push.  For energy consistency,  the update of the  canonical momentum  is modified by replacing the point-wise midpoint derivative of the vector potential with an orbit-averaged-discrete-gradient of the mesh-interpolated vector potential consistent with the orbit-average maps.  This construction satisfies an exact finite-difference chain rule along each particle orbit.  As a result, the particle work equals the mesh work appearing in the Crank--Nicolson field-energy balance, yielding exact total-energy conservation up to nonlinear solver tolerance and roundoff.  We demonstrate exact energy conservation of the method in 3D on the cold two-stream instability.
\end{abstract}



\section{Introduction}
\label{sec:introduction}

The purpose of this paper is to close a gap between two desirable features of
potential-based particle-in-cell (PIC) methods: exact preservation of the
Lorenz-gauge/Gauss-law structure and exact conservation of the discrete
electromagnetic energy.  The generalized-momentum formulation introduced in
\cite{ChristliebSandsWhite2025PartI} gives a natural way to evolve particles
using scalar and vector potentials without explicitly differencing the vector
potential in time.  The gauge-conserving developments in
\cite{ChristliebSandsWhite2024PartII,ChristliebSandsWhite2025PartIII} show that,
for properly time-consistent wave solvers, mapping current first and then
advancing charge from continuity gives a unstaggered method that preserves the
Lorenz gauge and Gauss's law.  The missing piece is an energy-conserving
particle coupling that is fully consistent with the same mesh-interpolated
vector potential used in the generalized momentum.

Since its inception, a central objective in the construction of robust PIC
methods has been to preserve the geometric and conservation structure of the
underlying plasma model.  As reviewed by Verboncoeur \cite{Verboncoeur2005},
PIC methods combine a Lagrangian particle representation of the distribution
function with an Eulerian field solve, and their long-time reliability depends
on the compatibility of interpolation, current deposition, field evolution, and
the particle pusher.  In electromagnetic PIC this has traditionally meant
preserving the Maxwell involutions, especially Gauss's law and the divergence constraints 
$\nabla\cdot\BB=0$, while using particle updates that respect the phase-space
structure of the Lorentz force.  Charge-conserving current deposition and
volume-preserving or symplectic particle pushers are examples of this broader
structure-preserving viewpoint \cite{BirdsallLangdon1985,HockneyEastwood1988,VillasenorBuneman1992,Esirkepov2001}.

Energy conservation has been a central concern since the early development of
kinetic plasma algorithms.  The work of Lewis \cite{Lewis1970} and Langdon
\cite{Langdon1973} already identified the connection between discrete
particle-field work, numerical heating, and long-time reliability.  The
implicit-PIC literature of the late 1970s and early 1980s developed this theme
further by coupling particles and fields implicitly in order to relax explicit
Debye-length, plasma-frequency, and finite-grid-instability restrictions
\cite{Mason1981,CohenLangdonFriedman1982,BrackbillForslund1982,LangdonCohenFriedman1983}.
These methods were major advances, but many were based on approximate or
linearized implicit couplings and therefore did not always yield a closed,
finite-time-step, fully discrete energy theorem.

A modern paradigm shift occurred around 2011, when exact energy conservation was
built directly into the nonlinear particle-field coupling.  In the electrostatic
Vlasov--Amp\`ere setting, Chen, Chac\'on, and Barnes
\cite{ChenChaconBarnes2011} introduced a fully implicit PIC method in which the
current in Amp\`ere's law is an orbit-averaged current computed from the same
particle trajectory used in the particle energy update.  When the nonlinear
system is converged, the field-energy change and particle-energy change cancel
as a fully discrete identity.  Markidis and Lapenta \cite{MarkidisLapenta2011}
developed a closely related fully implicit energy-conserving method in the same
period.  The essential design principle in these works is that the discrete
current and the discrete particle work are not independent approximations; they
are two representations of the same particle-field exchange.

Over the last fifteen years, this principle has been extended in several
directions.  Fully implicit orbit-averaged methods have been generalized to
mapped and body-fitted meshes, Vlasov--Darwin electromagnetic models,
curvilinear geometries, conducting boundaries, finite-grid-stability analyses,
and local energy conservation laws
\cite{ChaconChenBarnes2013,ChenChacon2014Darwin1D,ChenChacon2015DarwinMultiD,ChaconChen2016Curvilinear,ChaconChen2019Boundary,BarnesChacon2021FiniteGrid,ChaconChen2025Local}.
Semi-implicit methods pursue the same cancellation while avoiding a full
nonlinear particle solve.  Lapenta's ECsim method \cite{Lapenta2017ECSIM} uses a
mass-matrix coupling to obtain exact total-energy conservation, and later
semi-implicit schemes enforce both charge and energy conservation through
Gauss-law corrections, relativistic extensions, or compatible finite-element
couplings \cite{ChenToth2019GLECSIM,ChenEtAl2020RelativisticSemiImplicit,CamposPintoPages2022ChECSIM}.

A complementary line of work uses geometric, variational, and Hamiltonian
structure.  Variational electromagnetic PIC methods derive the discrete
particle-field system from an action principle, so that discrete gauge symmetry
implies a discrete Gauss law \cite{SquireQinTang2012,EvstatievShadwick2013}.
Noncanonical symplectic PIC and the GEMPIC framework preserve the Hamiltonian
or Poisson structure, the Maxwell constraints, and associated Casimirs at the
semidiscrete or fully discrete level
\cite{XiaoEtAl2015,KrausEtAl2017GEMPIC}.  Energy-conserving time propagation for
GEMPIC and related compatible finite-element formulations shows how the
geometric viewpoint can be combined with exact Hamiltonian preservation
\cite{KormannSonnendrucker2021,CamposPintoKormannSonnendrucker2022}.  More
recent explicit or nearly explicit approaches, including relativistic
$\pi$-PIC and explicit constrained energy-conserving updates, show that exact
energy balance can also be enforced by modified splittings or local work
constraints \cite{Gonoskov2024PiPIC,RicketsonHu2025Explicit}.

The present paper extends this energy-conserving viewpoint to the unstaggered
generalized-momentum PIC framework of Christlieb, Sands, and White.  Part I
introduced the potential-based generalized-momentum formulation for PIC in the
Lorenz gauge \cite{ChristliebSandsWhite2025PartI}.  Part II showed how the
Lorenz gauge can be enforced on co-located meshes by making the potential solve
compatible with the discrete continuity equation
\cite{ChristliebSandsWhite2024PartII}.  Part III developed a family of
unstaggered gauge-conserving methods and clarified the chain by which
discrete continuity gives Lorenz-gauge preservation, and Lorenz-gauge
preservation gives Gauss's law \cite{ChristliebSandsWhite2025PartIII}.  Those
papers established the potential formulation, the co-located field
representation, and the gauge/Gauss preservation mechanism.  The present work
adds the missing energy-conserving particle-field coupling by extending the
framework to a fully implicit, orbit-averaged formulation.

The main observation is simple.  The usual midpoint discretization of the
canonical momentum equation suggests evaluating $\grad\AV$ at a particle
midpoint.  This is second-order accurate, but it does not satisfy the
finite-difference chain rule for the mesh-interpolated vector potential
$\AV_h(x)=\sum_g \AV_g S_g(x)$ along a finite particle orbit.  The energy proof
needs exactly this chain rule.  We therefore replace the pointwise midpoint
derivative by an orbit-discrete-gradient.  The new object is a particle-side
derivative of the same interpolant $\AV_h$ used in $\PP=m\vv+q\AV_h$.  It is not
a new field solve, and it does not move the vector potential off the mesh.
Rather, it changes how a particle samples the mesh field along its own orbit.

The proposed method is an unstaggered Crank--Nicolson PIC scheme with a single
compatible particle-field exchange.  First, the current is deposited to the mesh
using orbit-averaged particle shapes.  Second, the charge density is advanced
from the discrete continuity equation.  Third, the scalar and vector potentials
are advanced by a Crank--Nicolson wave solve.  Finally, particles are advanced
in canonical momentum using an orbit-discrete-gradient of $\AV_h$.  When the
coupled nonlinear system is solved to tolerance, the discrete particle-energy
change is  the negative of the discrete field-energy change.  Thus the
method combines the continuity--gauge--Gauss structure of the unstaggered
potential formulation with the orbit-averaged work identity of modern energy-conserving~PIC.

This paper is organized around the main mathematical points.  \Cref{sec:model}
gives the non-dimensional, nonrelativistic potential model.
\Cref{sec:cn-gauge-gauss} introduces the unstaggered Crank--Nicolson field update
and states the gauge/Gauss preservation result.  \Cref{sec:energy-push}
contains the energy proof and the new particle push.  \Cref{sec:numerics}
presents numerical experiments that verify the identities and demonstrate the method. 
\Cref{sec:conclusion} gives a short summary.

\section{Nondimensional nonrelativistic potential model}
\label{sec:model}

This section fixes the model and notation.  The formulation follows the potential and generalized-momentum framework of \cite{ChristliebSandsWhite2025PartI}, but we restrict attention to the nonrelativistic setting in order to focus on the energy-conservation mechanism.  The nondimensionalization uses the same type of scales as in Appendix A of \cite{ChristliebSandsWhite2025PartI}: $x=L\tilde x$, $t=T\tilde t$, $v=(L/T)\tilde v$, $\phi=\phi_0\tilde\phi$, $\AV=A_0\tilde\AV$, with $\phi_0=ML^2/(QT^2)$ and $A_0=ML/(QT)$.  We drop tildes in what follows.

\subsection{Potential form of Maxwell's equations}
\label{subsec:potential-model}

Let $\phi$ be the scalar potential and $\AV$ be the vector potential.  In the Lorenz gauge, the nondimensional potential equations may be written as

\begin{subequations}
\label{eq:model-potentials}
\begin{align}
  \frac{1}{\kappaC^2}\partial_{tt}\phi - \Delta \phi &= \sigma_1 \rho,\label{eq:model-phi}\\
  \frac{1}{\kappaC^2}\partial_{tt}\AV - \Delta \AV &= \sigma_2 \JJ,\label{eq:model-A}\\
  \frac{1}{\kappaC^2}\partial_t\phi + \divg \AV &=0.\label{eq:model-gauge}
\end{align}
\end{subequations}
Here $\kappaC=cT/L$ is the nondimensional speed of light, while $\sigma_1$ and $\sigma_2$ are nondimensional source coefficients determined by the scaling.  For Debye-length and plasma-period normalizations used in the earlier generalized-momentum work, one often has $\sigma_1=1$ and $\sigma_2=\kappaC^{-2}$ in the electromagnetic wave equation followed by Maxwell scaling
\begin{equation}\label{eq:Maxwell_scaling}
    \sigma_1 = \kappaC^2 \sigma_2 .
\end{equation}
The algebraic conservation argument does not depend on this particular choice, but only on using Maxwell-compatible coefficients consistently in the field update and in the corresponding field energy. Other nondimensionalizations require the corresponding modifications to the field-energy weights and the Amp\`ere equation.

The electromagnetic fields are recovered from the potentials by
\begin{equation}
  \EE=-\grad\phi-\partial_t\AV,
  \qquad
  \BB=\curl\AV.
  \label{eq:E-B-cont}
\end{equation}
This representation gives $\divg \BB=0$ by construction.

\subsection{Canonical momentum particles}
\label{subsec:canonical-particles}

For a particle with mass $m_i$ and charge $q_i$, define the canonical momentum
\begin{equation}
  \PP_i=m_i\vv_i+q_i\AV(\xx_i,t).
  \label{eq:model-P-def}
\end{equation}
The nonrelativistic Hamiltonian is
\begin{equation}
  H_i(\xx_i,\PP_i,t)=\frac{1}{2m_i}\left|\PP_i-q_i\AV(\xx_i,t)\right|^2+q_i\phi(\xx_i,t).
  \label{eq:particle-H}
\end{equation}
Hamilton's equations give
\begin{subequations}
\label{eq:model-particles}
\begin{align}
  \dot{\xx}_i &= \vv_i = \frac{1}{m_i}\left(\PP_i-q_i\AV(\xx_i,t)\right),\label{eq:model-xdot}\\
  \dot{\PP}_i &= -q_i\grad\phi(\xx_i,t) + q_i\left(\grad\AV(\xx_i,t)\right)^T\vv_i.\label{eq:model-Pdot}
\end{align}
\end{subequations}
The key point is that the particle force contains only spatial derivatives of the potentials.  The time derivative of $\AV$ is not explicitly differenced in the particle push.  This was one of the motivations for the generalized-momentum PIC formulation in \cite{ChristliebSandsWhite2025PartI}.

\subsection{Mesh representation and source moments}
\label{subsec:mesh-representation}

Let $g$ denote a mesh index and let $S_g(\xx)$ be the periodic particle shape associated with mesh node or cell center $g$.  The mesh-interpolated vector potential seen by a particle is
\begin{equation}
  \AV_h^n(\xx)=\sum_g \AV_g^n S_g(\xx).
  \label{eq:Ah-mesh}
\end{equation}
The same shape functions are used to transfer particle information to the mesh and mesh information back to the particles.  In the original PIC formulation, source moments have the form
\begin{equation}
  \rho(\xx,t)=\sum_i q_i S(\xx-\xx_i(t)),
  \qquad
  \JJ(\xx,t)=\sum_i q_i\vv_i(t)S(\xx-\xx_i(t)).
  \label{eq:source-moments}
\end{equation}
In the method below, the current is deposited with orbit-averaged shapes, and the charge density is advanced from the continuity equation rather than deposited directly.  This is the same source-ordering principle used in the gauge-conserving methods of \cite{ChristliebSandsWhite2024PartII,ChristliebSandsWhite2025PartIII}.

\section{Unstaggered  field update and the gauge/Gauss structure}
\label{sec:cn-gauge-gauss}

This section states the field update independently of the particle push.  The main point is that the gauge and Gauss constraints are controlled by time consistency: the same temporal centering and the same spatial divergence used in the potential update must be used in the continuity update. Thus, charge conservation is used
as the mechanism that propagates the gauge and Gauss constraints. This was the central idea in the gauge-conserving papers \cite{ChristliebSandsWhite2024PartII,ChristliebSandsWhite2025PartIII}.  The presentation here is specialized to the Crank--Nicolson (CN) first-order potential system used for energy conservation.

\subsection{First-order potential system}
\label{subsec:first-order-potential-system}

Introduce
\begin{equation}
  \psi=\partial_t\phi,
  \qquad
  \UU=\partial_t\AV.
\end{equation}
Then the potential equations are written as
\begin{subequations}
\label{eq:first-order-system}
\begin{align}
  \partial_t\phi &= \psi,\label{eq:fo-phi}\\
  \partial_t\psi &= \kappaC^2\Delta\phi + \kappaC^2\sigma_1\rho,\label{eq:fo-psi}\\
  \partial_t\AV &= \UU,\label{eq:fo-A}\\
  \partial_t\UU &= \kappaC^2\Delta\AV + \kappaC^2\sigma_2\JJ.\label{eq:fo-U}
\end{align}
\end{subequations}
The electric and magnetic fields are evaluated from the evolved variables by
\begin{equation}
  \EE=-\grad\phi-\UU,
  \qquad
  \BB=\curl\AV.
  \label{eq:E-B-evolved}
\end{equation}

\subsection{Crank--Nicolson discretization}
\label{subsec:cn-discretization}

For any quantity $f$, define $f^{n+1/2}=(f^{n+1}+f^n)/2$.  The unstaggered CN update is
\begin{subequations}
\label{eq:cn-field-update}
\begin{align}
  \frac{\phi^{n+1}-\phi^n}{\dt} &= \psi^{n+1/2},\label{eq:cn-phi}\\
  \frac{\psi^{n+1}-\psi^n}{\dt} &= \kappaC^2\Delta_h\phi^{n+1/2}+\kappaC^2\sigma_1\rho^{n+1/2},\label{eq:cn-psi}\\
  \frac{\AV^{n+1}-\AV^n}{\dt} &= \UU^{n+1/2},\label{eq:cn-A}\\
  \frac{\UU^{n+1}-\UU^n}{\dt} &= \kappaC^2\Delta_h\AV^{n+1/2}+\kappaC^2\sigma_2\JJ^{n+1/2}.
  \label{eq:cn-U}
\end{align}
\end{subequations}
The midpoint Lorenz residual is
\begin{equation}
  g^{n+1/2}=\frac{1}{\kappaC^2}\psi^{n+1/2}+\divgh \AV^{n+1/2}.
  \label{eq:gauge-residual}
\end{equation}
The current is deposited first, and the charge is then advanced by
\begin{equation}
  \frac{\rho^{n+1}-\rho^n}{\dt}+\divgh \JJ^{n+1/2}=0.
  \label{eq:cn-continuity}
\end{equation}
This is the fully discrete analogue of the source treatment in \cite{ChristliebSandsWhite2024PartII,ChristliebSandsWhite2025PartIII}: the charge density is determined by continuity rather than by a direct charge scatter.

\subsection{Fourier form and constraint propagation}
\label{subsec:constraint-propagation}

For periodic problems, we use Fourier spectral derivatives.  For a fixed nonzero mode $\kk$, set
\begin{equation}
  \divgh \mapsto \ii \kk\cdot,
  \qquad
  \Delta_h \mapsto -|\kk|^2.
\end{equation}
Define the modal Lorenz and Gauss residuals
\begin{equation}
  g^n=\frac{1}{\kappaC^2}\Psi^n+\ii\kk\cdot \AV^n,
  \label{eq:modal-g}
\end{equation}
and
\begin{equation}
  r^n=\ii\kk\cdot \EE^n-\sigma_1 R^n,
  \qquad
  \EE^n=-\ii\kk\Phi^n-\UU^n.
  \label{eq:modal-r}
\end{equation}
A direct substitution of the CN update and the continuity update gives
\begin{subequations}
\label{eq:g-r-oscillator}
\begin{align}
  \frac{g^{n+1}-g^n}{\dt} &= -r^{n+1/2},\label{eq:g-evol}\\
  \frac{r^{n+1}-r^n}{\dt} &= \kappaC^2|\kk|^2 g^{n+1/2}.
  \label{eq:r-evol}
\end{align}
\end{subequations}
Thus the constraint pair $(g,r)$ is advanced by a homogeneous Crank--Nicolson discretization of an oscillator.  The source terms cancel by the consistency of the charge and vector-potential updates: both use the same current and discrete divergence operator, together with the Maxwell-compatible scaling relation \eqref{eq:Maxwell_scaling}.

\begin{theorem}[Gauge and Gauss preservation]
\label{thm:gauge-gauss}
Assume the scaling relation \eqref{eq:Maxwell_scaling}, periodic boundary conditions and Fourier spectral spatial operators.  Suppose the discrete continuity equation \eqref{eq:cn-continuity} is satisfied, and suppose the initial data satisfy the Lorenz and Gauss constraints mode by mode.  Then the CN update \eqref{eq:cn-field-update} preserves the Lorenz gauge and Gauss's law for every resolved Fourier mode, up to nonlinear solver tolerance and roundoff.
\end{theorem}

\begin{proof}
For each nonzero Fourier mode, the residuals $g^n$ and $r^n$ satisfy \eqref{eq:g-r-oscillator}.  This system is linear and homogeneous.  Therefore $(g^0,r^0)=(0,0)$ implies $(g^n,r^n)=(0,0)$ for every later time level.  For the zero mode, the periodic problem requires net-neutral charge and a fixed mean for the scalar potential.  The magnetic Gauss law is satisfied because $\BB=\curlh\AV$ and $\divgh\curlh=0$ in the spectral discretization.
\end{proof}

The theorem is not new in spirit; it is the CN specialization of the time-consistency idea developed in \cite{ChristliebSandsWhite2024PartII,ChristliebSandsWhite2025PartIII}.  The important point for the present paper is that the field update already gives the correct gauge/Gauss structure.  The remaining issue is the particle coupling needed for energy conservation.

\section{Energy conservation and the orbit-discrete-gradient  push}
\label{sec:energy-push}

This section contains the key new mathematical development in this work.  The method
presented thus far inherits the main structural properties of the previous unstaggered
potential formulations: the current is mapped to the mesh, the charge is then advanced by a
consistent discretization of the continuity equation, and this current-then-continuity
construction gives the time consistency needed to preserve the Lorenz gauge and Gauss's
law.  That observation is the foundation of the gauge-conserving methods in the earlier
papers.  Our goal here is to add one more structure, namely exact discrete conservation of
total energy.

The new point is that energy conservation requires more than a time-consistent field solve
and more than using the same particle shape in scatter and gather.  In the generalized
momentum formulation, the particle force contains only spatial derivatives of the potentials.
At the continuous level this is possible because
\begin{equation}
  \frac{\dd}{\dd t}\AV(\xx_i(t),t)
  =\partial_t\AV(\xx_i(t),t)
  +\grad \AV(\xx_i(t),t) \dot{\xx}_i(t).
  \label{eq:continuous-A-chain-rule}
\end{equation}
Equivalently, along the particle orbit,
\begin{equation}
  \AV(\xx_i^{n+1},t^{n+1})-\AV(\xx_i^n,t^n)
  =\int_{t^n}^{t^{n+1}}\partial_t\AV(\xx_i(t),t)\dd t
  +\int_{t^n}^{t^{n+1}}\grad\AV(\xx_i(t),t)\,\vv_i(t)\dd t .
  \label{eq:continuous-A-integral-chain-rule}
\end{equation}
In the discrete potential method we evolve the auxiliary variable
$\UU=\partial_t\AV$, and the Crank--Nicolson update provides the midpoint identity
\begin{equation}
  \frac{\AV_g^{n+1}-\AV_g^n}{\dt}=\UU_g^{n+1/2},
  \qquad
  \UU_g^{n+1/2}=\frac{1}{2}\left(\UU_g^{n+1}+\UU_g^n\right).
  \label{eq:U-midpoint-definition-intro}
\end{equation}
Thus no new staggered value of $\UU$ is introduced.  The time part of
\eqref{eq:continuous-A-integral-chain-rule} is already built into the CN field update.  The
spatial part is more delicate.  A pointwise midpoint approximation to
$\int \grad\AV(\xx_i(t),t)\vv_i(t)\dd t$ is second-order accurate, but it does not in
general satisfy the finite-difference chain rule for the mesh-interpolated field sampled by
the particle.  This is the missing piece in a direct midpoint implementation.

The purpose of this section is to build a particle push that makes the discrete chain rule an
identity.  We first define the total discrete energy.  We then introduce an orbit-averaged
scatter/gather operator, because the same particle orbit must be used both in the current map
and in the field gather.  Next we define an orbit-discrete-gradient of the mesh-interpolated
vector potential.  This derivative is not a gathered spectral gradient; it is the derivative of the
same interpolant that appears in the generalized momentum relation.  With this construction,
the particle energy balance reduces exactly to the mesh work term.  The CN field update
then gives the opposite field-energy balance, and total energy conservation follows.

\subsection{Discrete energy}
\label{subsec:discrete-energy}

Let $S_g(\xx)$ denote the particle shape associated with mesh point $g$.  The vector
potential seen by a particle is the mesh interpolant \eqref{eq:Ah-mesh}
\begin{equation*}
  \AV_h^n(\xx)=\sum_g \AV_g^n S_g(\xx).
\end{equation*}
This is the same gather used in the definition of the generalized momentum.  At integer time
levels we write
\begin{equation}
  m_i\vv_i^n=\PP_i^n-q_i\AV_h^n(\xx_i^n).
  \label{eq:discrete-canonical-relation}
\end{equation}
The particle kinetic energy is
\begin{equation}
  K^n=\sum_i \frac{1}{2}m_i|\vv_i^n|^2.
  \label{eq:particle-energy}
\end{equation}
The field energy is defined by
\begin{equation}
  W^n=
  \sum_g
  \left(
    \frac{1}{2\sigma_1}|\EE_g^n|^2
    +
    \frac{1}{2\sigma_2}|\BB_g^n|^2
  \right)\dV,
  \label{eq:field-energy}
\end{equation}
where
\begin{equation}
  \EE_g^n=-\grad_h\phi_g^n-\UU_g^n,
  \qquad
  \BB_g^n=\curlh\AV_g^n.
  \label{eq:fields-from-potentials-discrete}
\end{equation}
The constants $\sigma_1$ and $\sigma_2$ are the nondimensional coefficients appearing in the
field equations.  The total energy is
\begin{equation}
  \Etot^n=K^n+W^n.
  \label{eq:total-discrete-energy}
\end{equation}
The exact discrete theorem $\Etot^{n+1}=\Etot^n$ is stated for exact orbit integrals; in computations, the residual is limited by solver tolerance, quadrature error, and roundoff.

\subsection{The particle orbit and the orbit-averaged shape}
\label{subsec:orbit-shape}

Let the particle displacement over one time step be
\begin{equation}
  \Delta \xx_i=\xx_i^{n+1}-\xx_i^n,
  \qquad
  \vbar_i
  =\frac{1}{2}\left(\vv_i^{n+1}+\vv_i^n\right).
  \label{eq:orbit-vbar}
\end{equation}
In a periodic domain, $\Delta\xx_i$ is understood as an unwrapped displacement.  The
position is wrapped back to the periodic mesh only when evaluating shape functions.  The
straight orbit used by the time-centered particle update is
\begin{equation}
  \xx_i(s)=\xx_i^n+s\Delta \xx_i,
  \qquad 0\le s\le 1.
  \label{eq:particle-orbit}
\end{equation}
For each particle and mesh point we define the orbit-averaged shape (similarly to what was done in \cite{chen2023implicit} in the strongly magnetized context):
\begin{equation}
  \Sbar_{ig}=\int_0^1 S_g(\xx_i(s))\dd s.
  \label{eq:orbit-average-shape}
\end{equation}
The current used in the field solve is then
\begin{equation}
  \JJ_g^{n+1/2}=\frac{1}{\dV}\sum_i q_i\vbar_i\Sbar_{ig}.
  \label{eq:orbit-current}
\end{equation}
This is the current associated with the particle orbit, not an average of two endpoint
currents.

The same orbit weights are used to gather midpoint mesh quantities.  For any mesh vector
field $\bm{G}_g$, define
\begin{equation}
  \overline{\bm{G}}_i=\sum_g \bm{G}_g\Sbar_{ig}.
  \label{eq:orbit-gather-general}
\end{equation}
Then the scatter and gather operators are adjoint in the only sense needed by the energy
proof:
\begin{equation}
  \sum_i q_i\vbar_i\cdot\overline{\bm{G}}_i
  =
  \sum_g \JJ_g^{n+1/2}\cdot\bm{G}_g\dV.
  \label{eq:deposit-gather-identity}
\end{equation}
This identity is the mesh version of particle work.  It will be applied with
$\bm{G}_g=\EE_g^{n+1/2}$.  The exact first-order orbit weights for scatter and gather are in Appendix A.   

\subsection{The discrete chain rule for the mesh-interpolated vector potential}
\label{subsec:chain-rule}

The orbit average in \eqref{eq:orbit-average-shape} fixes the power balance between
particles and the mesh.  We now fix the other part of the energy argument, namely the chain
rule for $\AV_h$ along the orbit.  Define the linear interpolation of the mesh values in the
step by
\begin{equation}
  \AV_g(s)=(1-s)\AV_g^n+s\AV_g^{n+1}.
  \label{eq:A-g-s}
\end{equation}
The mesh-interpolated vector potential along the particle path is
\begin{equation}
  \mathcal A_i(s)=\sum_g \AV_g(s)S_g(\xx_i(s)).
  \label{eq:Acal-s}
\end{equation}
The CN field update gives
\begin{equation}
  \AV_g^{n+1}-\AV_g^n=\dt\,\UU_g^{n+1/2},
  \qquad
  \UU_g^{n+1/2}=\frac{1}{2}\left(\UU_g^{n+1}+\UU_g^n\right).
  \label{eq:A-U-CN}
\end{equation}
Using the same orbit weights as above, define
\begin{equation}
  \Ubar_i=\sum_g \UU_g^{n+1/2}\Sbar_{ig}.
  \label{eq:Ubar-def}
\end{equation}
The new  object is the orbit-discrete-gradient matrix $\DA_i\in\mathbb R^{d\times d}$,
with entries
\begin{equation}
  [\DA_i \AV]_{\ell j}
  =
  \sum_g \int_0^1
  A_{\ell,g}(s)\,\partial_{x_j}S_g(\xx_i(s))\dd s.
  \label{eq:DA-def}
\end{equation}
Here $\ell$ indexes the vector component of $\AV$ and $j$ indexes the spatial derivative.
This is the derivative of the particle interpolant \eqref{eq:Ah-mesh} along the orbit.

\begin{lemma}[Exact orbit chain rule]
\label{lem:exact-chain-rule}
The definitions \eqref{eq:orbit-average-shape}, \eqref{eq:Ubar-def}, and \eqref{eq:DA-def}
imply
\begin{equation}
  \AV_h^{n+1}(\xx_i^{n+1})-
  \AV_h^n(\xx_i^n)
  =
  \dt\,\Ubar_i+
  \dt\,(\DA_i\AV)\vbar_i.
  \label{eq:exact-chain-rule}
\end{equation}
\end{lemma}

\begin{proof}
Differentiate \eqref{eq:Acal-s} with respect to $s$:
\begin{equation}
  \frac{\dd \mathcal A_i}{\dd s}
  =
  \sum_g (\AV_g^{n+1}-\AV_g^n)S_g(\xx_i(s))
  +
  \sum_g \AV_g(s)\grad S_g(\xx_i(s)) \cdot \Delta\xx_i.
\end{equation}
Using \eqref{eq:A-U-CN} and $\Delta\xx_i=\dt\,\vbar_i$, and integrating from $s=0$ to
$s=1$, gives \eqref{eq:exact-chain-rule}.
\end{proof}

This lemma is the main construction.  It replaces the chain-rule assumption in the formal
energy argument by an identity of the particle-mesh discretization.  Notice that it does not
change the mesh field.  The field $\AV_g$ still lives on the mesh and is advanced by the CN
field solve.  What changes is the particle-side derivative used in the momentum equation.  For a first order particle weighting, Appendix B gives the exact orbit averaged gradient.   

\subsection{The energy-consistent canonical momentum update}
\label{subsec:energy-consistent-push}
We now explain how we derive  the update for \eqref{eq:model-particles}, the canonical momentum, along the discreet  orbit \eqref{eq:orbit-vbar}.  
As noted in the last section, the purpose of the orbit-discrete-gradient construction is to replace the
pointwise quantity $\nabla\AV(\xx_i,t)$ 
by a particle-orbit derivative of the same mesh interpolant that appears in the generalized momentum relation with an approximation consistent with the discreet chain rule along the orbit.  As discussed in the last section, this distinction is important.   For the force coming from the scalar potential, we use a gather that makes use of the same orbit-averaged shape as the current deposition:
\begin{equation}
  \gphibar_i
  =
  \sum_g
  \left(\nabla_h\phi_g^{n+1/2}\right)
  \overline S_{ig},
  \qquad
  \overline S_{ig}
  =
  \int_0^1 S_g(\xx_i(s))\,ds .
  \label{eq:gphibar-definition}
\end{equation}
For the gradient of the vector potential, 
 the consistent orbit averaged component that appears in the $j$-th equation for the canonical
momentum is
$$
\begin{aligned}
   ((\nabla \AV)^T\vv_i)_{ j} \approx  \left[
    \left(\DA_i \AV\right)^T
    \vbar_i
    \right]_j
    &=
    \sum_{\ell=1}^d
    \left[
    \left(\DA_i\AV\right)^T
    \right]_{j\ell}
    \overline v_{i,\ell} =
    \sum_{\ell=1}^d
    \left[
    \DA_i\AV
    \right]_{\ell j}
    \overline v_{i,\ell} \\
    &=
    \sum_{\ell=1}^d
    \sum_g
    \int_0^1
    A_{\ell,g}(s)\,
    \partial_{x_j}S_g(\xx_i(s))\,ds\,
    \overline v_{i,\ell}.
\end{aligned}
$$
The resulting orbit-discrete-gradient particle push is
\begin{subequations}
\label{eq:orbit-push}
\begin{align}
  \vbar_i
  &=
  \frac{1}{2}\left(\vv_i^{n+1}+\vv_i^n\right),
  \label{eq:orbit-vbar-update}\\
  \xx_i^{n+1}
  &=
  \xx_i^n+\dt\,\vbar_i,
  \label{eq:orbit-x-update}\\
  \PP_i^{n+1}
  &=
  \PP_i^n+
  \dt\left[-q_i\gphibar_i+q_i(\DA_i \AV)^T\vbar_i\right],
  \label{eq:orbit-P-update}\\
  m_i\vv_i^{n+1}
  &=
  \PP_i^{n+1}-q_i\AV_h^{n+1}(\xx_i^{n+1}).
  \label{eq:orbit-v-update}
\end{align}
\end{subequations}
This is not the same as using a pointwise midpoint gather of $\grad\AV$.  The matrix
$\DA_i \AV$ is built from the derivative of the same interpolant $\AV_h$ that appears in the
canonical momentum relation.  This consistency is what makes the discrete chain rule exact and will give point wise energy conservation.

\subsection{Particle energy balance}
\label{subsec:particle-energy-balance}

Subtract \eqref{eq:discrete-canonical-relation} at two time levels:
\begin{equation}
  m_i(\vv_i^{n+1}-\vv_i^n)
  =
  (\PP_i^{n+1}-\PP_i^n)
  -q_i\left[
    \AV_h^{n+1}(\xx_i^{n+1})-
    \AV_h^n(\xx_i^n)
  \right].
  \label{eq:mechanical-momentum-difference}
\end{equation}
Dot with $\vbar_i$ and sum over particles.  Since
\begin{equation}
  m_i\vbar_i\cdot(\vv_i^{n+1}-\vv_i^n)
  =
  \frac{m_i}{2}\left(|\vv_i^{n+1}|^2-|\vv_i^n|^2\right),
\end{equation}
we obtain
\begin{equation}
  K^{n+1}-K^n
  =
  \sum_i \vbar_i\cdot(\PP_i^{n+1}-\PP_i^n)
  -
  \sum_i q_i\vbar_i\cdot
  \left[
    \AV_h^{n+1}(\xx_i^{n+1})-
    \AV_h^n(\xx_i^n)
  \right].
  \label{eq:kinetic-change-pre-chain}
\end{equation}
We next rewrite \eqref{eq:kinetic-change-pre-chain} in a form close to what we need to establish energy conservation by
using the momentum update \eqref{eq:orbit-P-update} and the discrete chain rule \eqref{eq:exact-chain-rule}.  What we are doing is replacing $(\PP_i^{n+1}-\PP_i^n)$ with the right hand side of the particle push as well as replacing  $\left(
    \AV_h^{n+1}(\xx_i^{n+1})-
    \AV_h^n(\xx_i^n)
  \right)$ 
with the right hand side of the discrete chain rule.  Making use of the following identity allows for the cancellation the terms $\Delta t \vbar_i\cdot (\DA_i {\bf A})^T\vbar_i$ coming from the discreet chain rule and  $\Delta t \vbar_i\cdot\DA_i \vbar_i$ coming form the particle push,
\begin{equation}
  \vbar_i\cdot (\DA_i \AV)^T\vbar_i
  =
  ((\DA_i \AV)\vbar_i)\cdot\vbar_i
  =
  \vbar_i\cdot(\DA_i \AV)\vbar_i~,~
  \label{eq:DA-cancellation}
\end{equation}
leaves us with the following, that the difference in the total kinetic energy is 
\begin{equation}
  K^{n+1}-K^n
  =
  \dt\sum_i q_i\vbar_i\cdot\left(-\gphibar_i-\Ubar_i\right).
  \label{eq:particle-energy-work-particle}
\end{equation}
If we define the orbit-averaged electric field as
\begin{equation}
  \overline{\EE}_i
  =
  \sum_g \EE_g^{n+1/2}\Sbar_{ig}
  =
  -\gphibar_i-\Ubar_i,
  \label{eq:orbit-E-gather}
\end{equation}
then
\begin{equation}
  K^{n+1}-K^n
  =
  \dt\sum_i q_i\vbar_i\cdot\overline{\EE}_i.
  \label{eq:particle-energy-work-particle-final}
\end{equation}
Using the deposit/gather identity \eqref{eq:deposit-gather-identity}, the particle energy
change is written entirely on the mesh:
\begin{equation}
  K^{n+1}-K^n
  =
  \dt\sum_g \JJ_g^{n+1/2}\cdot\EE_g^{n+1/2}\dV.
  \label{eq:particle-energy-final}
\end{equation}
This definition for the orbit-averaged electric field was not arbitrary, it is motivated by the orbit averaged particle map for $\JJ$ and the CN solution for the fields.  The next section completes the construction by showing the  differences of the total mesh field energy  at time $t^{n+1}$ and $t^n$ is the right hand side of equation \ref{eq:particle-energy-final}, which will allow us to establish $\Etot^{n+1}=\Etot^n$ at a fully discreet level.

\subsection{Field energy balance and total energy conservation}
\label{subsec:field-energy-balance}

The CN potential update and the midpoint Lorenz gauge imply the CN Maxwell form
\begin{subequations}
\label{eq:cn-maxwell}
\begin{align}
  \frac{\BB^{n+1}-\BB^n}{\dt}
  &=-\curlh\EE^{n+1/2},
  \label{eq:cn-faraday}\\
  \frac{1}{\sigma_1}\frac{\EE^{n+1}-\EE^n}{\dt}
  &=
  \frac{1}{\sigma_2}\curlh\BB^{n+1/2}
  -\JJ^{n+1/2}.
  \label{eq:cn-ampere}
\end{align}
\end{subequations}
The derivation uses the scaling relation \eqref{eq:Maxwell_scaling}, the definitions \eqref{eq:fields-from-potentials-discrete},  the CN updates for $(\phi,\psi,\AV,\UU)$, and the midpoint Lorenz gauge constraint.  Taking the
discrete inner product of \eqref{eq:cn-ampere} with $\EE^{n+1/2}$ and
\eqref{eq:cn-faraday} with $\BB^{n+1/2}/\sigma_2$, adding, and using summation by parts
for the periodic spectral operators gives
\begin{equation}
  W^{n+1}-W^n
  =
  -\dt\sum_g \JJ_g^{n+1/2}\cdot\EE_g^{n+1/2}\dV.
  \label{eq:field-energy-final}
\end{equation}
Combining \eqref{eq:particle-energy-final} and \eqref{eq:field-energy-final} gives the
main conservation statement.

\begin{theorem}[Discrete total energy conservation]
\label{thm:energy-conservation}
Assume periodic boundary conditions, summation-by-parts spectral operators, the continuity
update \eqref{eq:cn-continuity}, the orbit current \eqref{eq:orbit-current}, the
orbit-discrete-gradient particle push \eqref{eq:orbit-push}, and a fully converged coupled
CN nonlinear solve.  If the orbit integrals are evaluated exactly, then
\begin{equation}
  \Etot^{n+1}=\Etot^n.
\end{equation}
In a computation, the equality holds up to nonlinear solver tolerance, quadrature error in the
orbit integrals, and floating-point roundoff.  Exact first order orbit weights are in Appendix A.
\end{theorem}

\subsection{Cell boundaries and conservation}
\label{sec:cell-boundaries-conservation}

The preceding proof assumes that the orbit integrals defining $\Sbar_{ig}$, $\Ubar_i$, and
$\DA_i\AV$ are exact.  This is an implementation point, but it is not a minor one.  A B-spline is
a compactly supported piecewise polynomial.  Along a particle path the active polynomial
branch changes whenever the path crosses a spline knot.  A Gauss rule applied over the whole
interval $0\le s\le 1$ is exact only if the path stays in one polynomial patch.  If a particle
crosses a cell boundary, or more precisely a spline knot, then the path integral must be split
at the crossing.  This is also true if one does the orbit integral exact.  

A one-dimensional test makes this visible.  Prescribe $A^n$ and $A^{n+1}$ on a periodic
mesh and measure only the chain-rule residual
\begin{equation}
  R_{A,i}
  =
  A_h^{n+1}(x_i^{n+1})-
  A_h^n(x_i^n)
  -
  \dt\left(\overline U_i+ \DA_i\AV \overline v_i\right).
  \label{eq:one-d-RA-diagnostic}
\end{equation}
This test does not use a field solve.  It only checks whether the particle orbit quadrature
satisfies the finite-difference identity in Lemma \ref{lem:exact-chain-rule}.  Table
\ref{tab:one-d-cell-crossing-linear} compares one Gauss rule over the full path with the
same Gauss rule applied after splitting at the crossed knot.  Once the path crosses the knot,
the unsplit rule is no longer exact.  The split-path rule restores the identity to roundoff.

\begin{table}[htbp]
\centering
\caption{
One-dimensional chain-rule residual for a linear spline when the particle crosses a spline knot. The unsplit full-path quadrature treats the piecewise-polynomial integrand as a single polynomial and produces a residual that increases with crossing size, whereas split-path quadrature restores the chain-rule identity to roundoff.
}
\label{tab:one-d-cell-crossing-linear}
\begin{tabular}{c c c}
\hline
Crossing size in cell widths & Full-step quadrature & Split-path quadrature \\
\hline
$0.01$ & $3.24\times 10^{-4}$ & $1.27\times 10^{-16}$ \\
$0.10$ & $3.24\times 10^{-3}$ & $6.11\times 10^{-16}$ \\
\hline
\end{tabular}
\end{table}

The same test also explains why increasing the quadrature order alone is not the correct
fix.  If the path is not split, the integrand is not a single polynomial on the quadrature
interval.  Smoother splines reduce the defect, but they do not remove its source.

\begin{table}[htbp]
\centering
\caption{
One-dimensional chain-rule residual for different spline degrees when the particle crosses a spline knot by $0.10$ cell widths.
Higher spline degree reduces the unsplit full-path quadrature error, but the chain-rule identity is recovered to roundoff only by splitting the path at the crossed knot.
}
\label{tab:one-d-cell-crossing-degree}
\begin{tabular}{c c c}
\hline
Spline degree & Unsplit full-path quadrature & Split-path quadrature \\
\hline
$1$ & $3.24\times 10^{-3}$ & $6.11\times 10^{-16}$ \\
$2$ & $2.40\times 10^{-6}$ & $1.69\times 10^{-15}$ \\
$3$ & $5.87\times 10^{-8}$ & $4.44\times 10^{-16}$ \\
\hline
\end{tabular}
\end{table}

For the multidimensional method, let ${\cal K}_\alpha$ be the set of spline knots in
coordinate direction $\alpha$.  For particle $i$, construct the breakpoint set
\begin{equation}
  {\cal B}_i
  =
  \{0,1\}
  \cup
  \left\{
    \frac{\xi-x_{i,\alpha}^n}{\Delta x_{i,\alpha}}
    :
    \xi\in {\cal K}_\alpha,
    \quad
    0<\frac{\xi-x_{i,\alpha}^n}{\Delta x_{i,\alpha}}<1,
    \quad
    \alpha=1,2,3
  \right\}.
  \label{eq:breakpoint-set}
\end{equation}
If $\Delta x_{i,\alpha}=0$, that coordinate contributes no breakpoints.  After sorting and
removing duplicates, write
\begin{equation}
  0=s_0<s_1<\cdots<s_{M_i}=1.
\end{equation}
The orbit integrals are then evaluated piecewise.  For example,
\begin{equation}
  \Sbar_{ig}
  =
  \sum_{m=0}^{M_i-1}
  \int_{s_m}^{s_{m+1}} S_g(\xx_i(s))\dd s,
  \label{eq:split-Sbar}
\end{equation}
and
\begin{equation}
  [\DA_i \AV]_{\ell j}
  =
  \sum_{m=0}^{M_i-1}\sum_g
  \int_{s_m}^{s_{m+1}}
  A_{\ell,g}(s)\,\partial_{x_j}S_g(\xx_i(s))\dd s.
  \label{eq:split-DA}
\end{equation}

\begin{algorithm}[H]
\caption{Picard iteration for the orbit-discrete-gradient CN potential PIC step}
\label{alg:picard-orbit-cn}
\begin{algorithmic}[1]
\Require Particle data $(\xx_i^n,\vv_i^n,\PP_i^n)$, history velocity $\vv_i^{n-1}$, and mesh data $(\phi^n,\psi^n,\AV^n,\UU^n,\rho^n)$.
\Ensure Updated particle and mesh data at time level $n+1$.
\State Initialize the end-of-step velocity by extrapolation:
\Statex \hspace{2em}$\vv_i^{n+1,0}=2\vv_i^n-\vv_i^{n-1}$.
\For{$p=0,1,2,\ldots,p_{\max}$}
  \State Form the orbit implied by the current iterate:
  \Statex \hspace{2em}$\vbar_i^{\,p}=\frac{1}{2}(\vv_i^{n+1,p}+\vv_i^n)$,
  \qquad
  $\xx_i^{n+1,p}=\xx_i^n+\dt\,\vbar_i^{\,p}$.
  \State For each particle, build ${\cal B}_i^p$ from all spline-knot crossings along the unwrapped path.
  \State Evaluate $\Sbar_{ig}^{\,p}$ and $\DA_i^{\,p}$ using split-path quadrature on the intervals in ${\cal B}_i^p$.
  \State Deposit the orbit current
  \Statex \hspace{2em}$\displaystyle \JJ_g^{n+1/2,p}=\frac{1}{\dV}\sum_i q_i\vbar_i^{\,p}\Sbar_{ig}^{\,p}$.
  \State Update charge from continuity:
  \Statex \hspace{2em}$\rho^{n+1,p}=\rho^n-\dt\,\divgh\JJ^{n+1/2,p}$,
  \qquad
  $\rho^{n+1/2,p}=\frac{1}{2}(\rho^{n+1,p}+\rho^n)$.
  \State Solve the CN potential equations for $(\phi^{n+1,p},\psi^{n+1,p},\AV^{n+1,p},\UU^{n+1,p})$ using $(\rho^{n+1/2,p},\JJ^{n+1/2,p})$.
  \State Form midpoint mesh fields, for example
  \Statex \hspace{2em}$\AV^{n+1/2,p}=\frac{1}{2}(\AV^{n+1,p}+\AV^n)$,
  \qquad
  $\UU^{n+1/2,p}=\frac{1}{2}(\UU^{n+1,p}+\UU^n)$.
  \State Gather $\gphibar_i^{\,p}$ and $\Ubar_i^{\,p}$ with the same orbit weights $\Sbar_{ig}^{\,p}$.
  \State Advance generalized momentum:
  \Statex \hspace{2em}$\PP_i^{n+1,p}=\PP_i^n+
  \dt\left[-q_i\gphibar_i^{\,p}+q_i(\DA_i^{\,p}\AV)^T\vbar_i^{\,p}\right]$.
  \State Update the end-of-step velocity:
  \Statex \hspace{2em}$\displaystyle
  \vv_i^{n+1,p+1}=\frac{1}{m_i}\left[
  \PP_i^{n+1,p}-q_i\AV_h^{n+1,p}(\xx_i^{n+1,p})
  \right]$.
  \State Compute the fixed-point errors in $\JJ^{n+1/2,p}$ and $\vv_i^{n+1,p+1}$.
  \State Optionally compute the chain-rule diagnostic
  \Statex \hspace{2em}$\displaystyle
  R_{A,i}^{\,p}=\AV_h^{n+1,p}(\xx_i^{n+1,p})-
  \AV_h^n(\xx_i^n)-
  \dt\left(\Ubar_i^{\,p}+(\DA_i^{\,p}\AV)\vbar_i^{\,p}\right)$.
  \If{the current and velocity fixed-point errors are below tolerance}
    \State Accept the iterate and exit the Picard loop.
  \EndIf
\EndFor
\State Set $(\xx_i^{n+1},\vv_i^{n+1},\PP_i^{n+1})=(\xx_i^{n+1,p},\vv_i^{n+1,p+1},\PP_i^{n+1,p})$ and keep the corresponding mesh fields.
\end{algorithmic}
\end{algorithm}

On each subinterval the integrand lies in a single tensor-product polynomial patch.  For a
spline of degree $r$ in three dimensions, the degree of the integrands in these orbit terms is
bounded by $3r$ on each split segment.  Thus a Gauss rule with $n_q$ points satisfying
\begin{equation}
  2n_q-1\ge 3r
\end{equation}
is sufficient for exact integration on each segment.

This explains a failure mode that is otherwise confusing.  The continuity update can keep the
Lorenz gauge and Gauss's law near roundoff even when the orbit integrals are only
approximate, because those properties are determined by the mesh current and charge.  Energy
conservation is more sensitive.  If the path is not split at spline knots, the code solves a
nearby but different particle problem, and the residual \eqref{eq:one-d-RA-diagnostic}
appears as accumulated energy drift.

\subsection{Picard iteration used in the numerical experiments}
\label{sec:picard-iteration-crossings}

We close this section by giving the nonlinear iteration used in the numerical experiments.
The iteration is intentionally simple.  Its role is to demonstrate the conservative
particle-field coupling, not to provide the most efficient nonlinear solver.  The conservation
statement is tied to the discrete equations being solved, not to whether those equations are
solved by Picard, Newton, Newton--Krylov, or an accelerated fixed point method.

There are two practical points in Algorithm \ref{alg:picard-orbit-cn}.  First, the current
used in the field solve is the same orbit current used in the deposit/gather power identity.
Second, every orbit integral is split at the spline knots crossed by the particle during the
time step.  Without this split, the residual $R_{A,i}$ is not a roundoff-level diagnostic, and
exact total energy conservation should not be expected in long runs.

\section{Numerical experiments}
\label{sec:numerics}

This section provides the numerical evidence for the method.  The purpose of the first set of tests is not to show a complicated plasma calculation, but to verify each identity used in the proof.  These tests are used as diagnostics for the overall code when running the standard PIC benchmarks such as the two-stream instability.  We have applied the method to classic test cases of weak and strong Landau damping, and 
method has performed well.  
The code and all of these test cases can all be found in the GitHub repo \url{https://github.com/sgong11/Unstaggered_PIC}.  The code is a teaching code, meaning it is a serial C++ code meant to be a clear explanation of the method.

\subsection{Gauge and Gauss diagnostics}
\label{subsec:gauge-gauss-tests}

For the gauge and Gauss diagnostics, we compute
\begin{equation}
  G_h^n=\frac{1}{\kappaC^2}\psi^n+\divgh\AV^n,
  \label{eq:gauge-diagnostic}
\end{equation}
and
\begin{equation}
  R_h^n=\divgh\EE^n-\sigma_1\rho^n.
  \label{eq:gauss-diagnostic}
\end{equation}
The tests should show that these quantities remain at solver tolerance when the continuity update is solved with the same spectral divergence used in the field update.  If these diagnostics drift, the likely causes are: the continuity equation is not being solved at the same time centering, the field equations are not converged, or different discrete derivative operators are being mixed.  These residuals confirm that we are conserving the Lorenz gauge and Gauss's law in our unstaggered formulation over the duration of the run.  

\subsection{Unit tests for the particle-mesh identities}
\label{subsec:unit-tests}

We introduce the orbit chain-rule residual, discussed in section \ref{sec:cell-boundaries-conservation}, to keep track of the errors introduced in the energy along the particle trajectory, 
\begin{equation}
  R_{A,i}=\AV_h^{n+1}(\xx_i^{n+1})-\AV_h^n(\xx_i^n)-\dt\left(\Ubar_i+(\DA_i \AV)\vbar_i\right).
  \label{eq:RA-test}
\end{equation}
For the orbit-discrete-gradient construction, $R_{A,i}$ should be zero up to quadrature and roundoff.  This is the most direct test that demonstrates the need for the new particle push used in this formulation.

The second test checks the deposit/gather power identity
\begin{equation}
  \sum_i q_i\vbar_i\cdot\overline{\EE}_i=
  \sum_g\JJ_g^{n+1/2}\cdot\EE_g^{n+1/2}\dV.
  \label{eq:power-test}
\end{equation}
This residual should also be at roundoff if the same orbit-averaged shapes are used for current deposition and field gather.

\subsection{Energy diagnostics}
\label{subsec:energy-diagnostics}

For every step we will report the split energy residuals
\begin{align}
  R_{\rm field} &= W^{n+1}-W^n+\dt\sum_g\JJ_g^{n+1/2}\cdot\EE_g^{n+1/2}\dV,\label{eq:R-field}\\
  R_{\rm particle} &= K^{n+1}-K^n-\dt\sum_iq_i\vbar_i\cdot\overline{\EE}_i,\label{eq:R-particle}\\
  R_{\rm total} &= (K^{n+1}+W^{n+1})-(K^n+W^n).
  \label{eq:R-total}
\end{align}
The orbit-discrete-gradient push should drive $R_{\rm particle}$ to quadrature and solver tolerance.
These are effective metrics for checking whether the nonlinear solve has been converged sufficiently.

\subsection{Cold two-stream instability}
\label{subsec:two-stream}


The  numerical test we consider as a demonstration is a three-dimensional version of the cold two-stream
instability.  The purpose of this example is not only to verify that the method
captures the expected roll-up of the two electron streams, but also to stress the
orbit integration used in the energy-conserving particle push.  In particular, this
test is designed so that a nontrivial number of particles cross cell boundaries
during the run.  This makes it a useful test for distinguishing the conservative
split-orbit implementation from an implementation that applies quadrature over the
full particle path without splitting at spline knots.

The computational domain is periodic in all directions and is given by
\[
    \Omega = [0,L_x]\times [0,L_y]\times [0,L_z],
    \qquad
    L_x=L_y=L_z=2\pi .
\]
The mesh is uniform with
\[
    N_x=N_y=N_z=16,
    \qquad
    \Delta x=\Delta y=\Delta z=\frac{2\pi}{16}.
\]
The time step is
\[
    \Delta t = 2.5\times 10^{-2},
\]
and the simulation is run for
\[
    N_t = 1200
\]
time steps.  Thus the final time is
\[
    T = N_t\Delta t = 30.
\]
All quantities are reported in the nondimensional variables used in the code.

The plasma consists of mobile electrons and a fixed neutralizing ion background.
The background density is
\[
    n_0 = 1.
\]
The normalized constants are
\[
    \kappaC =1,\qquad \sigma_1=1,\qquad \sigma_2=1.
\]
Only the electrons are represented by particles.  The ions enter as a spatially
uniform positive background charge.  The electron macroparticles are initialized as
two equal counterstreaming populations with drift velocities $+v_0$ and $-v_0$ in
the $x$ direction, where
\[
    v_0 = 0.3 .
\]
The perturbation is introduced in the particle velocity, following the standard
cold two-stream setup:
\[
    v_{x,i}^{+}(0)
    =
    v_0 + \epsilon \sin(kx_i),
    \qquad
    v_{x,i}^{-}(0)
    =
    -v_0 + \epsilon \sin(kx_i),
\]
where
\[
    \epsilon = 2.0\times 10^{-2},
    \qquad
    k = \frac{2\pi m}{L_x},
    \qquad
    m=1 .
\]
The transverse velocities are initialized to zero,
\[
    v_{y,i}^{\pm}(0)=v_{z,i}^{\pm}(0)=0.
\]

The particle locations are chosen so that the initial distribution is uniform in
space.  In the $x$ direction, particles are placed on a linearly spaced
cell-centered lattice.  The same $x$ lattice is used for each $y$-$z$ cell, and
the two streams are placed at identical positions.  This gives a clean
three-dimensional realization of a one-dimensional two-stream perturbation.  If
$N_{\mathrm{ppc}}$ denotes the number of particle pairs per mesh cell, then the
number of electron macroparticles is
\[
    N_p = 2 N_{\mathrm{ppc}} N_xN_yN_z .
\]
For this test,
\[
    N_{\mathrm{ppc}} = 16,
\]
so that
\[
    N_p = 2(16)(16^3) = 131072.
\]
The charge of each electron macroparticle is chosen so that the mobile electron
charge balances the uniform ion background:
\[
    q_p = -\frac{n_0|\Omega|}{N_p},
    \qquad
    m_p = |q_p|.
\]
Thus $q_p/m_p=-1$ for each electron macroparticle.

The initial charge density is formed by depositing the electron charge on the mesh
and adding the uniform ion background.  The mean charge is removed so that the
periodic zero mode is exactly neutral:
\[
    \rho^0_g
    =
    \rho^0_{e,g}
    +
    n_0
    -
    \left\langle
    \rho^0_e+n_0
    \right\rangle .
\]
The initial current is deposited from the paired electron streams.  Since the
streams are located at the same positions and have equal and opposite drift
velocities, the initial current is zero up to the perturbation and roundoff.

The initial scalar potential is obtained from the periodic Poisson equation
\[
    -\nabla_h^2 \phi^0 = \frac{\rho^0}{\epsilon_0},
\]
with mean-zero $\phi^0$.  The remaining initial field variables are
\[
    \psi^0 = 0,
    \qquad
    {\bf A}^0 = {\bf 0},
    \qquad
    {\bf U}^0 = {\bf 0}.
\]
The initial generalized momentum is therefore
\[
    {\bf P}_i^0
    =
    m_i{\bf v}_i^0
    +
    q_i{\bf A}_h^0({\bf x}_i^0)
    =
    m_i{\bf v}_i^0 .
\]

The orbit integrals are evaluated using tensor-product linear B-spline particle
shapes,
\[
    r=1,
\]
and a Gauss rule with
\[
    n_q=16
\]
points on each orbit segment.  As the Gauss rule is exact in this case, the answer is identical with the analytical first order weights in Appendix A.  The conservative method splits each particle orbit
at all crossed spline knots before applying the quadrature rule.  To demonstrate
why this is necessary, we compare against an otherwise identical run in which the
quadrature is applied over the full particle path without splitting.  In the input
deck this is controlled by
\[
    \texttt{split\_orbit\_at\_knots}
    =
    \begin{cases}
    \texttt{true}, & \text{conservative split-orbit method},\\
    \texttt{false}, & \text{unsplit comparison run}.
    \end{cases}
\]

\begin{figure}
    \centering
    \includegraphics[width=0.7\linewidth]{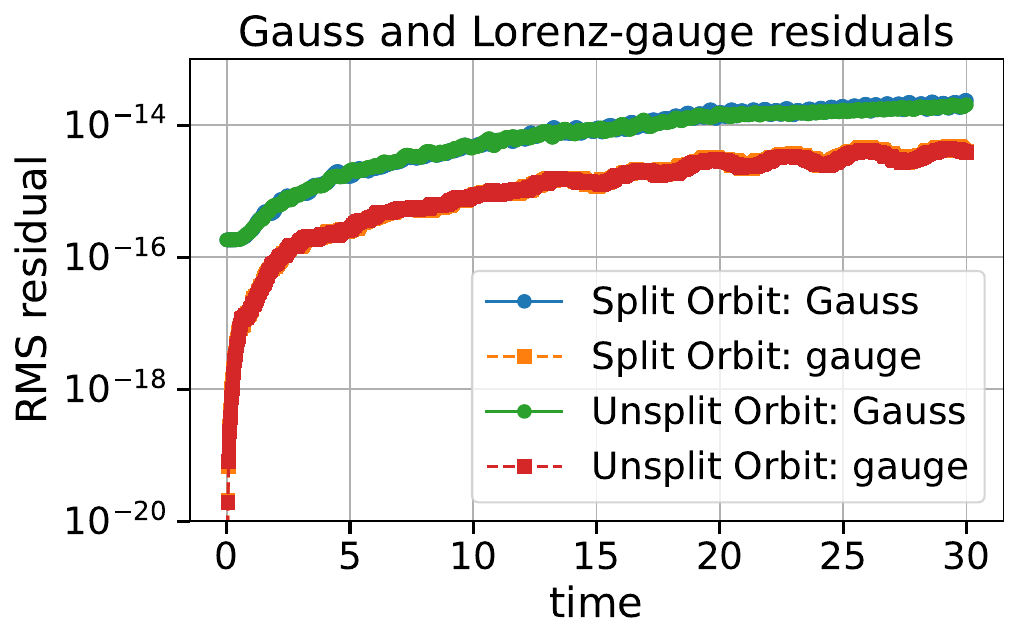}
    \caption{Gauge and Gauss residuals in the 3D cold two-stream test. Both split- and unsplit-orbit variants keep the Lorenz-gauge residual $G^n_h$ and Gauss-law residual $R_h^n$ at approximately roundoff/solver tolerance through T = 30. This shows that involution preservation is controlled by the compatible current- continuity-field update, not by the orbit splitting used for energy conservation.}
    \label{fig:constraints}
\end{figure}

The nonlinear Picard iteration is stopped when both the midpoint current and the
particle velocity have converged.  The tolerance and maximum iteration count are
\[
    \tau_{\mathrm{Picard}} = 10^{-11},
    \qquad
    p_{\max}=16.
\]
A final consistency sweep is applied after convergence.  This means that the final
diagnostics are evaluated using the same current, fields, orbit integrals, and
particle state that define the accepted time step.

\begin{figure}[htbp]
    \centering
    \begin{subfigure}[b]{0.48\linewidth}
        \centering
        \includegraphics[width=\linewidth]{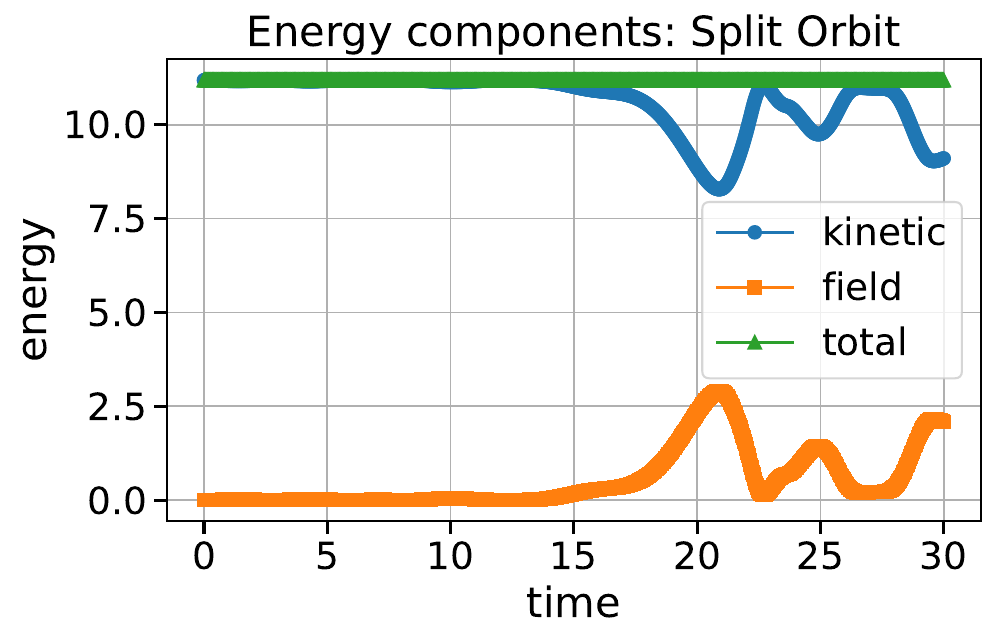}
        \caption{Energy exchange in the split-orbit method.}
        \label{fig:components_primary}
    \end{subfigure}
    \hfill
    \begin{subfigure}[b]{0.48\linewidth}
        \centering
        \includegraphics[width=\linewidth]{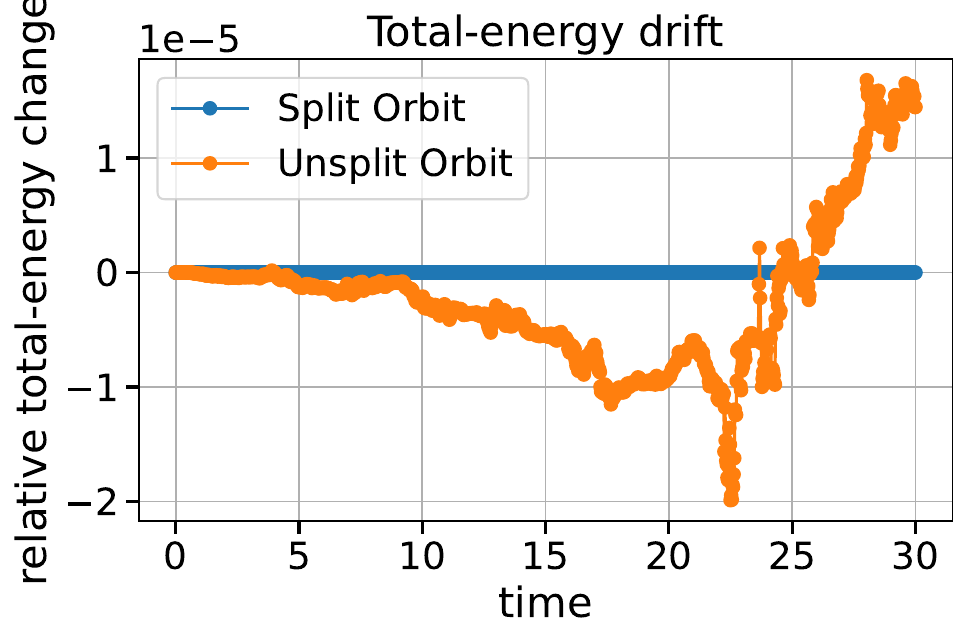}
        \caption{Comparison of total energy preservation.}
        \label{fig:drift_compare}
    \end{subfigure}

    \caption{
    Energy behavior for split- and unsplit-orbit integration.
    (a) In the split-orbit method, energy is exchanged between kinetic and field components while the total energy remains constant. Here, green denotes total energy, orange denotes potential energy, and blue denotes kinetic energy.
    (b) Relative energy drift as a function of time for split and unsplit orbits. The split-orbit method preserves total energy to roundoff accuracy, whereas the unsplit-orbit method develops an $\mathcal{O}(10^{-5})$ relative drift, consistent with an orbit chain-rule defect.
    }
    \label{fig:energy_components_primary}
\end{figure}

In Figure \ref{fig:constraints}  we plot the Gauge and Gauss residuals for both the split and unsplit orbits. We see that both methods preserve the discreet Gauge condition and Gauss's Law to machine precision ion.  
This confirms that exact energy conservation does not impact the conservation of the involutions on the mesh.  Indeed, the original method has this property.  
Figure \ref{fig:energy_components_primary}(a) shows exact energy conservation for the method when orbits are split, showing both the kinetic, potential and the total. 
The relative energy drift,
\[
    \frac{{\cal E}^n-{\cal E}^0}{{\cal E}^0},
\]
for the split and unsplit particle orbit method is plotted in Figure \ref{fig:energy_components_primary} (b).  We see the that unsplit orbits have errors $10^{-5}$, where the split orbits preserve relative energy to machine precision.

Figure~\ref{fig:defect}(a) plots the field-energy and particle chain-rule defects as functions of time for the split-orbit method. We see that, for the split method, the residuals are maintained at or below the iteration tolerance. 
\begin{figure}[htbp]
    \centering
    \begin{subfigure}[b]{0.48\linewidth}
        \centering
        \includegraphics[width=\linewidth]{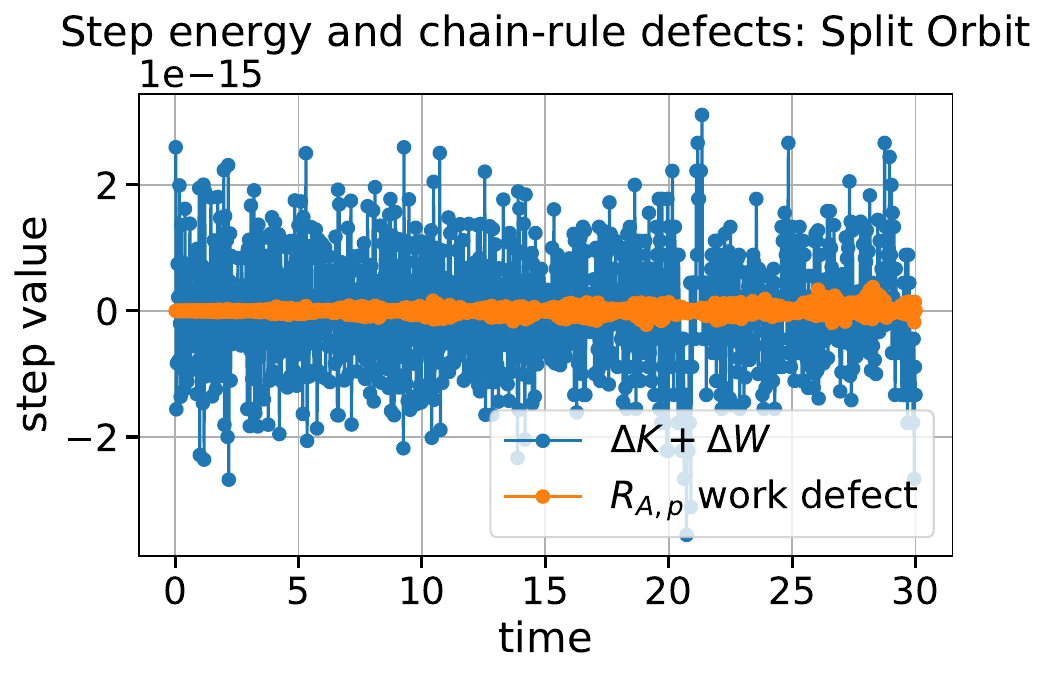}
        \caption{Step energy defect.}
        \label{fig:defect_primary}
    \end{subfigure}
    \hfill
    \begin{subfigure}[b]{0.48\linewidth}
        \centering
        \includegraphics[width=\linewidth]{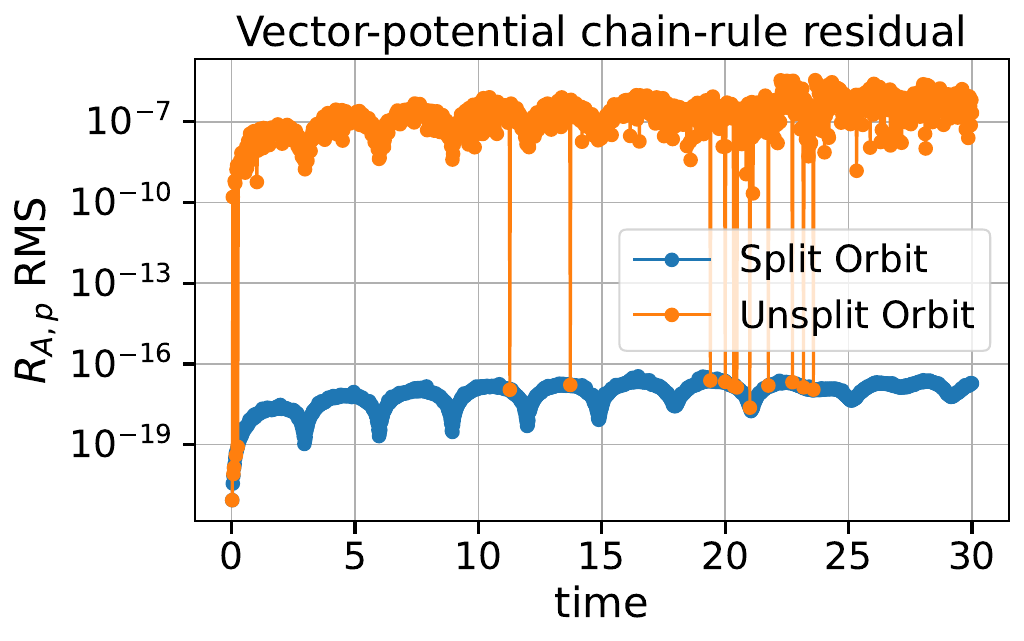}
        \caption{RMS vector-potential chain-rule residual.}
        \label{fig:rms_compare}
    \end{subfigure}

    \caption{
    Energy and vector-potential chain-rule defects.
    (a) For the split-orbit method, both the step energy defect $\Delta K + \Delta W$ and the particle work/chain-rule defect remain at the solver tolerance or roundoff level.
    (b) Splitting at spline knots reduces the RMS vector-potential chain-rule residual by approximately ten orders of magnitude relative to the unsplit rule, explaining the energy drift observed in Fig.~\ref{fig:drift_compare}.
    }
    \label{fig:defect}
\end{figure}
The energy diagnostics separate the three possible sources of error: particle-grid
work inconsistency, field energy inconsistency, and particle orbit inconsistency.
For the split-orbit method, all three residuals should remain at roundoff or
solver tolerance.  For the unsplit comparison, the gauge and Gauss residuals can
remain small while $R_{A,i}^n$ and $R_{\mathrm{part}}^n$ become nonzero.  
Figure \ref{fig:defect}(b) shows  the orbit chain-rule residual \eqref{eq:RA-test} for both the split and unsplit particle orbits. 
If the orbit integrals are evaluated consistently, then
\[
    R_{A,i}^n = 0
\]
up to quadrature error and roundoff.  If the path crosses spline knots and the
integral is not split at the crossings, then this residual becomes nonzero and
appears as a defect in the particle energy balance.
This diagnostic is central to the present test.  We see that the unsplit orbit integrator has 10 orders of magnitude larger error than the split particle orbits.  
It is clear that this method is the main contribution to the relative energy drift for the unsplit method. 
\begin{figure}[htbp]
    \centering
    \begin{subfigure}[b]{0.48\linewidth}
        \centering
        \includegraphics[width=\linewidth]{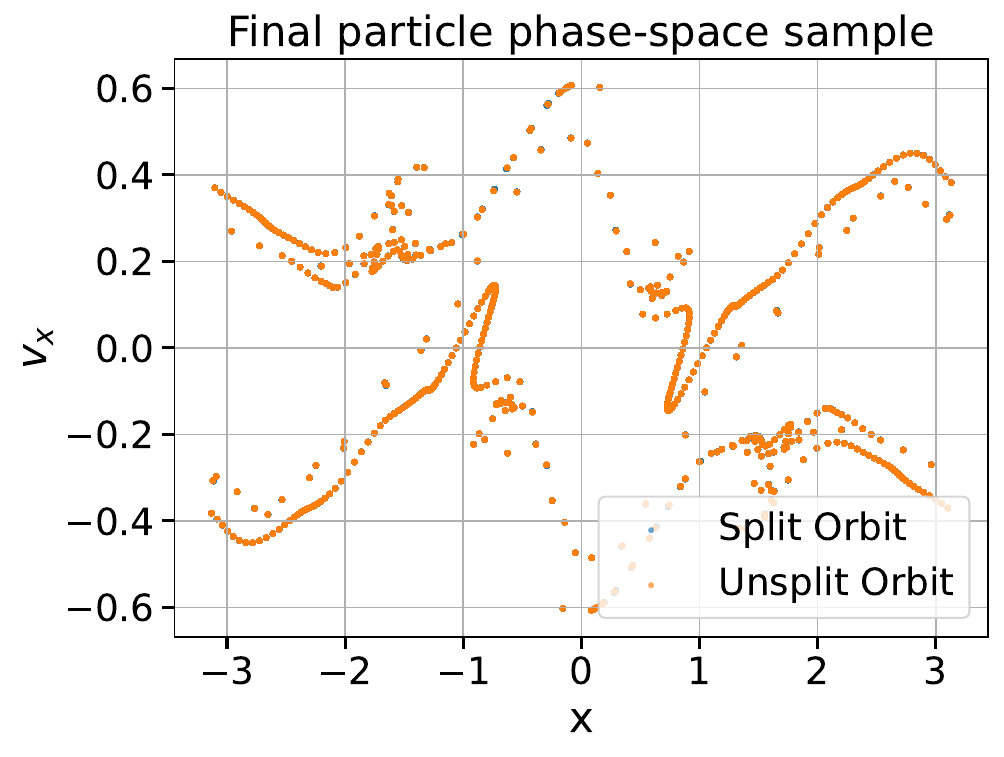}
        \caption{Phase-space slice at $t=30$.}
        \label{fig:phase_space_primary}
    \end{subfigure}
    \hfill
    \begin{subfigure}[b]{0.48\linewidth}
        \centering
        \includegraphics[width=\linewidth]{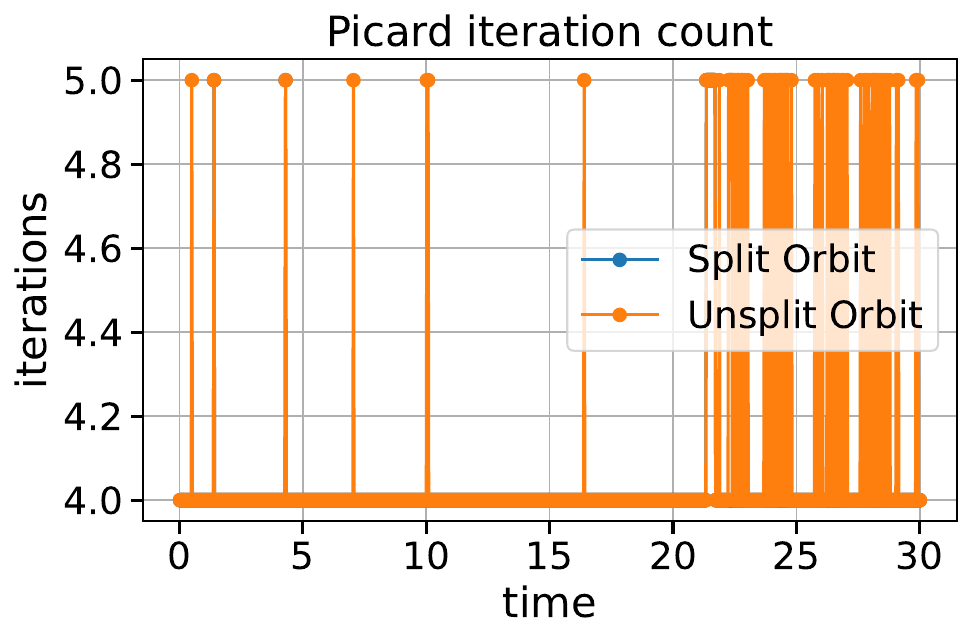}
        \caption{Picard iteration count.}
        \label{fig:picard_iterations_compare}
    \end{subfigure}
    \caption{
    Phase-space structure and nonlinear iteration cost.
    (a) Slice of the $x$--$v_x$ phase space at $T=30$ from the 3D simulation. The split- and unsplit-orbit implementations produce visually similar samples, so the conservation error is not evident from this diagnostic alone.
    (b) Picard iteration count as a function of time. For this test, the simple Picard iteration converges in approximately four to five iterations per step.
    }
    \label{fig:twostream}
\end{figure}
Figure~\ref{fig:twostream}(a) plots the phase space for the split and unsplit methods at time $t=30$. While there is not much visible difference between the two solutions, this is mainly because the energy drift is small compared with the $O(1)$ phase-space values.  

For this test problem, we also look at iteration count for the simple Picard iteration used in this work.  As in \cite{ChenChaconBarnes2011}, it is expected that a full Newton-Krylov method  would  reduce the associated iteration count.  Further,  pre-conditioning can greatly increase efficiency, and will be pursued in future versions of the method.  However the main point to this paper is to introduce the new energy conserving formulation. As such, this will be the subject of future work. Figure~\ref{fig:twostream}(b) shows the Picard iteration count at each time step.  It is observed that with this very simple formulation, the iteration count is reasonable.

\section{Conclusion}
\label{sec:conclusion}

This paper introduces an unstaggered energy-conserving particle-in-cell method.  We start by extending the Gauge/Gauss conserving explicit unstaggered PIC method to the CN continuity-based charge update and potential solver.  The remaining issue is the discrete particle work.  A pointwise midpoint approximation to $\grad\AV$ does not satisfy the finite-difference chain rule for the mesh-interpolated vector potential along a particle orbit.  The introduced orbit-discrete-gradient particle push fixes this by constructing the derivative of $\AV_h$ that makes the chain rule exact.

The resulting scheme has three important properties.  First, the particle update remains in canonical momentum and therefore avoids explicit time differencing of the vector potential in the particle force.  Second, the current and electric-field gather use the same orbit-averaged shape weights, so particle work equals mesh work.  Third, the orbit-discrete-gradient of $\AV_h$ cancels the vector-potential terms in the particle energy balance exactly.  Together with the CN field-energy balance, this gives total-energy conservation for the fully coupled nonlinear solve.

The numerical results confirm the central identities used in the analysis. In the three-dimensional cold two-stream test, the continuity-based charge update and Crank--Nicolson potential solve preserve the Lorenz gauge and Gauss's law to roundoff for both the split- and unsplit-orbit variants, showing that the constraint preservation is controlled by the compatible current-continuity-field update. The distinction appears in the energy diagnostics: when particle paths are split at spline knots, the orbit-discrete-gradient chain-rule residual remains at roundoff and the total energy is conserved to machine precision, whereas the unsplit orbit integration produces a nonzero chain-rule defect and a corresponding energy drift. These results demonstrate that exact energy conservation in the unstaggered potential formulation requires both the  orbit-averaged scatter/gather map and the orbit-discrete-gradient of the mesh-interpolated vector potential. Future work will extend the formulation to the relativistic Vlasov--Maxwell system, develop more efficient nonlinear solvers based on Anderson acceleration and Newton--Krylov methods with physics-based preconditioning, and build a scalable parallel implementation aimed at relativistic laser--plasma interaction problems.

\section*{Acknowledgments}
AC would like to acknowledge the support from AFOSR grants FA9550-24-1-0254, DOE grant DE-SC0023164, DOE/NNSA grant DE-NA0004265 and ONR grant N00014-24-1-2242. LC contribution was funded by the Office of Advanced Scientific Computing Research of the U.S. Department of Energy, and performed at Los Alamos National Laboratory under contract 89233218CNA000001. The authors acknowledge the use of ChatGPT for grammar checking.
\appendix

\section{Exact first-order orbit weights for scatter and gather}
\label{app:first_order_orbit_weights}

This appendix gives the analytic first-order orbit weights used in the
orbit-averaged scatter and gather. The key point is that these are path
integrals. They become one-dimensional polynomial integrals only after the
mesh shape has been pulled back to the particle orbit. Since a first-order
B-spline is piecewise linear, the path must first be split at every crossed
spline knot. On each split segment the active cell and the active
tensor-product stencil are fixed.

Let
$$
        d_i=x_i^{n+1}-x_i^n
$$
be the unwrapped particle displacement over one time step, and let
$$
        x_i(s)=x_i^n+s d_i, \qquad 0\le s\le 1 .
$$
Let
$$
        0=s_0<s_1<\cdots<s_{M_i}=1
$$
be the sorted list of all knot-crossing parameters, including the endpoints.
On one segment, set
$$
        h_m=s_{m+1}-s_m, \qquad \eta=s-s_m, \qquad 0\le \eta\le h_m .
$$
For mesh spacings $H_x,H_y,H_z$, introduce local normalized coordinates
$$
        r_\alpha(\eta)=r_{\alpha,m}+\beta_\alpha\eta,
        \qquad
        \beta_\alpha=\frac{d_{i,\alpha}}{H_\alpha},
        \qquad
        \alpha\in\{x,y,z\}.
$$
Here $r_{\alpha,m}$ is the normalized coordinate at the beginning of the
segment. The split construction guarantees that $0\le r_\alpha(\eta)\le 1$
on the segment, up to endpoint conventions at the knots.

For first-order weighting,
$$
        w_{\alpha,0}(\eta)=1-r_\alpha(\eta),
        \qquad
        w_{\alpha,1}(\eta)=r_\alpha(\eta).
$$
Write each one-dimensional weight as
$$
        w_{\alpha,\nu}(\eta)
        =
        \lambda_{\alpha,\nu}+\mu_{\alpha,\nu}\eta,
        \qquad \nu\in\{0,1\},
$$
where
$$
        \lambda_{\alpha,0}=1-r_{\alpha,m},\qquad
        \mu_{\alpha,0}=-\beta_\alpha,
        \qquad
        \lambda_{\alpha,1}=r_{\alpha,m},\qquad
        \mu_{\alpha,1}=\beta_\alpha .
$$

For a tensor-product offset $\nu=(\nu_x,\nu_y,\nu_z)\in\{0,1\}^3$, define
the local mesh point on segment $m$ by $g_\nu^{(m)}$. The pulled-back shape is
$$
        S_\nu(x_i(s_m+\eta))
        =
        w_{x,\nu_x}(\eta)
        w_{y,\nu_y}(\eta)
        w_{z,\nu_z}(\eta).
$$
Set
$$
        L_\alpha=\lambda_{\alpha,\nu_\alpha},
        \qquad
        M_\alpha=\mu_{\alpha,\nu_\alpha}.
$$
Then
$$
\begin{aligned}
        S_\nu(x_i(s_m+\eta))
        &=
        (L_x+M_x\eta)(L_y+M_y\eta)(L_z+M_z\eta)  \\
        &=
        c_0+c_1\eta+c_2\eta^2+c_3\eta^3,
\end{aligned}
$$
with
$$
\begin{aligned}
        c_0 &= L_xL_yL_z,\\
        c_1 &= M_xL_yL_z+L_xM_yL_z+L_xL_yM_z,\\
        c_2 &= M_xM_yL_z+M_xL_yM_z+L_xM_yM_z,\\
        c_3 &= M_xM_yM_z .
\end{aligned}
$$
Thus the exact segment contribution is
$$
        I_\nu^{(m)}
        =
        \int_{s_m}^{s_{m+1}} S_\nu(x_i(s))\,ds
        =
        c_0h_m+\frac{c_1}{2}h_m^2+\frac{c_2}{3}h_m^3+\frac{c_3}{4}h_m^4 .
$$
The orbit weight is accumulated by adding each segment contribution to its
corresponding global mesh point:
$$
        S_{i,g_\nu^{(m)}} \mathrel{+}= I_\nu^{(m)} .
$$
Equivalently,
$$
        S_{ig}
        =
        \sum_{m=0}^{M_i-1}
        \sum_{\nu\in\{0,1\}^3:\,g_\nu^{(m)}=g}
        I_\nu^{(m)} .
$$
Periodic wrapping is applied to the mesh index only after the local stencil
has been identified.

The orbit-averaged current is
$$
        J_g^{n+1/2}
        =
        \frac{1}{\Delta V}
        \sum_i q_i v_i S_{ig},
        \qquad
        v_i=\frac{x_i^{n+1}-x_i^n}{\Delta t}.
$$
The same weights are used for gather. For any mesh field $G_g$,
$$
        G_i
        =
        \sum_g G_g S_{ig}
        =
        \sum_{m=0}^{M_i-1}
        \sum_{\nu\in\{0,1\}^3}
        G_{g_\nu^{(m)}} I_\nu^{(m)} .
$$
This gives the scatter-gather work identity
$$
        \sum_i q_i v_i\cdot G_i
        =
        \sum_g J_g^{n+1/2}\cdot G_g\,\Delta V .
$$
In the energy proof this identity is used with $G_g=E_g^{n+1/2}$.

A useful local check is partition of unity. Since
$$
        w_{\alpha,0}(\eta)+w_{\alpha,1}(\eta)=1
$$
in each direction,
$$
        \sum_{\nu\in\{0,1\}^3} I_\nu^{(m)}=h_m .
$$
Therefore
$$
        \sum_g S_{ig}=1
$$
for every particle. Failure of this check usually indicates an error in the
knot crossings, stencil indices, or endpoint conventions.

\section{Exact first-order orbit-discrete gradients}
\label{app:first_order_orbit_gradients}

The energy-conserving canonical-momentum update also requires the orbit
average of the derivative of the same mesh interpolant used in
$P=mv+qA_h(x)$. This object is not the spectral gradient of $A$ gathered to
the particle. It is the orbit-discrete gradient of the interpolant
$$
        A_h(x)=\sum_g A_g S_g(x).
$$

Let the vector potential be linearly interpolated in the global orbit
parameter:
$$
        A_{\ell,g}(s)
        =
        (1-s)A_{\ell,g}^n+sA_{\ell,g}^{n+1}.
$$
On segment $m$, with $s=s_m+\eta$, define for the local mesh point
$g_\nu^{(m)}$
$$
        A_{\ell,\nu}^{(m)}
        =
        (1-s_m)A_{\ell,g_\nu^{(m)}}^n
        +
        s_m A_{\ell,g_\nu^{(m)}}^{n+1},
        \qquad
        C_{\ell,\nu}^{(m)}
        =
        A_{\ell,g_\nu^{(m)}}^{n+1}
        -
        A_{\ell,g_\nu^{(m)}}^n .
$$
Then
$$
        A_{\ell,g_\nu^{(m)}}(s_m+\eta)
        =
        A_{\ell,\nu}^{(m)}
        +
        C_{\ell,\nu}^{(m)}\eta .
$$
The dependence of $A_{\ell,\nu}^{(m)}$ on the global value $s_m$ is essential.

For first-order weights,
$$
        \frac{\partial w_{\alpha,0}}{\partial x_\alpha}
        =
        -\frac{1}{H_\alpha},
        \qquad
        \frac{\partial w_{\alpha,1}}{\partial x_\alpha}
        =
        \frac{1}{H_\alpha}.
$$
Define $\sigma_0=-1$ and $\sigma_1=1$. For derivative direction $j$, let $p$
and $q$ denote the two transverse coordinate directions. Then
$$
        \partial_{x_j}S_\nu(x_i(s_m+\eta))
        =
        \frac{\sigma_{\nu_j}}{H_j}
        w_{p,\nu_p}(\eta)w_{q,\nu_q}(\eta).
$$
The transverse product is quadratic:
$$
        w_{p,\nu_p}(\eta)w_{q,\nu_q}(\eta)
        =
        d_0+d_1\eta+d_2\eta^2,
$$
where
$$
        d_0=L_pL_q,
        \qquad
        d_1=M_pL_q+L_pM_q,
        \qquad
        d_2=M_pM_q .
$$
Thus the exact segment contribution to $[D_i]_{\ell j}$ is
$$
\begin{aligned}
        I_{\ell j,\nu}^{D,(m)}
        =
        \frac{\sigma_{\nu_j}}{H_j}
        \bigg[
        A_{\ell,\nu}^{(m)}
        \left(
        d_0h_m+\frac{d_1}{2}h_m^2+\frac{d_2}{3}h_m^3
        \right)
        +
        C_{\ell,\nu}^{(m)}
        \left(
        \frac{d_0}{2}h_m^2+\frac{d_1}{3}h_m^3+\frac{d_2}{4}h_m^4
        \right)
        \bigg].
\end{aligned}
$$
The orbit-discrete-gradient matrix is
$$
        [D_i\AV]_{\ell j}
        =
        \sum_{m=0}^{M_i-1}
        \sum_{\nu\in\{0,1\}^3}
        I_{\ell j,\nu}^{D,(m)} .
$$

With the same split segments and the same first-order orbit weights,
$$
        U_i
        =
        \sum_{m=0}^{M_i-1}
        \sum_{\nu\in\{0,1\}^3}
        U_{g_\nu^{(m)}}^{n+1/2} I_\nu^{(m)} .
$$
The definitions above satisfy the exact finite-difference chain rule
$$
        A_h^{n+1}(x_i^{n+1})-A_h^n(x_i^n)
        =
        \Delta t\,U_i+\Delta t\,(D_i\AV) v_i .
$$
This is the identity needed by the canonical-momentum energy balance. Since
the particle push uses $(D_i\AV)^T v_i$, the vector-potential terms cancel after
dotting with $v_i$:
$$
        v_i\cdot (D_i\AV)^T v_i
        =
        v_i\cdot (D_i\AV) v_i .
$$
The implementation diagnostic is
$$
        R_{A,i}
        =
        A_h^{n+1}(x_i^{n+1})
        -
        A_h^n(x_i^n)
        -
        \Delta t\left(U_i+(D_i\AV) v_i\right).
$$
For first-order weights with analytic split-path integration, $R_{A,i}$ should
be at roundoff. If a particle crosses a cell boundary and the integral is not
split at the crossed knot, the pulled-back integrand is not a single polynomial
on the integration interval, and this residual is generally nonzero.

\bibliographystyle{siamplain}
\bibliography{references}

\end{document}